\documentclass[12pt]{amsart}
\usepackage{fullpage,amssymb,url}
\makeatletter
\let\@@pmod\pmod
\DeclareRobustCommand{\pmod}{\@ifstar\@pmods\@@pmod}
\def\@pmods#1{\mkern4mu({\operator@font mod}\mkern 6mu#1)}
\makeatother
\newcommand{\Z}{\mathbb{Z}}
\newtheorem{theorem}{Theorem}
\newtheorem{proposition}[theorem]{Proposition}
\newtheorem{lemma}[theorem]{Lemma}
\begin{document}
\title{A variant of the Euclid--Mullin sequence containing every prime}
\author[A. R. Booker]{Andrew R.~Booker}
\thanks{The author was partially supported by EPSRC Grant {\tt EP/K034383/1}.}
\address{
Howard House\\
University of Bristol\\
Queens Ave\\
Bristol\\
BS8 1SN\\
United Kingdom
}
\email{\tt andrew.booker@bristol.ac.uk}
\begin{abstract}
We consider a generalization of Euclid's proof of the infinitude of primes
and show that it leads to variants of the Euclid--Mullin sequence that
provably contain every prime number.
\end{abstract}
\maketitle
\section{Introduction}
Given a finite set $\{p_1,\ldots,p_k\}$ of prime numbers, let $p_{k+1}$ be
a prime factor of $1+p_1\cdots p_k$. Then, as shown by Euclid, $p_{k+1}$
is necessarily distinct from $p_1,\ldots,p_k$. Iterating this procedure,
we thus obtain an infinite sequence of distinct primes. For instance,
beginning with $k=0$ (with the convention that the empty product is
$1$) and choosing $p_{k+1}$ as small as possible at each step, one
obtains the \emph{Euclid--Mullin sequence} (\texttt{A000945} in the
OEIS \cite{a000945}). More generally, following Clark \cite{clark},
we call any sequence resulting from this construction a
\emph{Euclid sequence with seed $\{p_1,\ldots,p_k\}$}.

One of the central questions in this area was posed by Mullin
\cite{mullin} in 1963: Does the Euclid--Mullin sequence contain every
prime number? Despite a compelling heuristic argument of Shanks \cite{shanks}
that the answer is yes, even the broader question of whether there
is \emph{any} Euclid sequence containing every prime number remains
open. (On the other hand, there are Euclid sequences that
provably do not contain every prime. For instance, starting from $k=0$
and choosing $p_{k+1}$ as large as possible at each step, one obtains the
\emph{second Euclid--Mullin sequence}, which is known to omit infinitely
many primes \cite{booker,pt}.)
In \cite{bi} it was shown that, for any given seed
$\{p_1,\ldots,p_k\}$, the possible Euclid sequences have a natural
directed graph structure. Although one can prove many interesting properties
of the family of graphs obtained by varying the seed, proving much
about any particular graph remains an elusive goal.

In this note, following a suggestion of Trevor Wooley, we consider a
generalization of Euclid's construction, in the hope that it will be more
amenable to proof. Precisely, if $\{p_1,\ldots,p_k\}$ is a set of prime
numbers, then for any $I\subseteq\{1,\ldots,k\}$, the number $N_I=\prod_{i\in
I}p_i+\prod_{i\in\{1,\ldots,k\}\setminus I}p_i$ is coprime to $p_1\cdots
p_k$ and has at least one prime factor. Iteratively choosing a set $I$ and
a prime $p_{k+1}\mid N_I$, we obtain an infinite sequence $p_1,p_2,\ldots$
of distinct prime numbers, as in Euclid's proof.  (Note that Euclid's
construction is the special case in which $I=\emptyset$ at each step.)

We call a sequence resulting from this more general construction
a \emph{generalized Euclid sequence with seed $\{p_1,\ldots,p_k\}$}.
Our result is that the construction is provably
general enough to obtain every prime.
\begin{theorem}\label{t:main}
For any finite set $P$ of prime numbers, there is a generalized Euclid
sequence with seed $P$ containing every prime.
\end{theorem}
One particular generalized Euclid sequence was defined by Chua
(\texttt{A167604} in the OEIS \cite{a167604}), starting with $k=0$
and choosing $p_{k+1}$ as small as possible at each step. A natural
question, analogous to Mullin's, is whether Chua's sequence itself
contains every prime. This seems very likely, but difficult to
prove, since there is an obstruction that prevents the terms from
always appearing in numerical order. Precisely, if $n=p_1\cdots p_k$
is the product of the first $k$ terms of Chua's sequence, then the
next term $p_{k+1}$ is the smallest prime factor of $\prod_{d\mid
n}(d+n/d)$; thus, $d^2+n\equiv0\pmod*{p_{k+1}}$ for some $d$, so that
$\left(\frac{-n}{p_{k+1}}\right)=1$. (Alekseyev has conjectured that
$p_{k+1}$ is always the smallest prime satisfying this constraint;
see \cite{a167604}.) Given the well-known difficulty of proving good
bounds for the gaps between sign changes of a quadratic character,
we cannot rule out the possibility that Chua's sequence is very thin.

We conclude the introduction by mentioning another variant of Euclid's
construction, due to Pomerance \cite[\S1.1.3]{cp}: given a set of
primes $\{p_1,\ldots,p_k\}$, let $p_{k+1}$ be a prime that is not one
of $p_1,\ldots,p_k$ and divides a number of the form $d+1$ for $d\mid
p_1\cdots p_k$. Then, starting from $k=0$ and choosing $p_{k+1}$ as
small as possible at each step, one obtains a sequence containing every
prime, and in fact $p_k$ is the $k$th smallest prime for $k\ge5$. While
our variant is arguably truer in spirit to Euclid's proof (since it is
guaranteed to produce only new primes at each step), Pomerance's variant
has the distinct advantage of exhibiting a specific sequence containing
every prime.

\subsection*{Acknowledgements}
I thank Trevor Wooley and Carl Pomerance for supportive comments.

\section{Proof of Theorem~\ref{t:main}}
Given a prime number $q$, let $S_q\subseteq(\Z/q\Z)^\times$ be the set
of residue classes attained by the squarefree, $(q-1)$-smooth,
positive integers, i.e.\
$$
S_q=\left\{d+q\Z:d\in\Z_{>0},\;d\mid\prod_{p<q}p\right\}.
$$
One of the main ingredients in the proof of Theorem~\ref{t:main}
is that $S_q$ is large, so that if $q$ is the smallest prime not yet
attained in $p_1,\ldots,p_k$, then there is a significant chance that
$q$ is a prime factor of $d+n/d$ for some $d\mid n=p_1\cdots p_k$. From
computation for small $q$, it seems likely that $S_q=(\Z/q\Z)^\times$
for all $q\notin\{5,7\}$. We are not aware of a proof of this, but it
turns out that the following weaker approximation is sufficient for
our purposes:
\begin{lemma}\label{l:Sq}
For any prime $q$, $\#S_q>\frac12(q-1)$.
\end{lemma}
\begin{proof}
For squarefree positive integers $d\le q-1$, the residue classes $d+q\Z$
are distinct and contained in $S_q$. By \cite{rogers}, the number of
such $d$ is at least $\frac{53}{88}(q-1)>\frac12(q-1)$.
\end{proof}
In addition, we need one further input from algebraic geometry:
\begin{lemma}\label{l:hyp}
Let $q$ be an odd prime number and $a\in(\Z/q\Z)^\times$.
\begin{itemize}
\item[(i)]If $q\ne5$ or $q=5$ and $a\ne3+5\Z$ then
there exists $x\in(\Z/q\Z)^\times$ such that
$\left(\frac{x+a/x}{q}\right)\ne1$.
\item[(ii)]If $q\notin\{7,13\}$ then
there exists $x\in(\Z/q\Z)^\times$ such that
$\left(\frac{x^6+a}{q}\right)\ne1$.
\end{itemize}
\end{lemma}
\begin{proof}
We consider the sum
$$
\sum_{x\in(\Z/q\Z)^\times}\left(\frac{x+a/x}{q}\right)
=\sum_{x\in(\Z/q\Z)^\times}\left(\frac{x(x^2+a)}{q}\right).
$$
For $q\ge3$, $x(x^2+a)$ has no repeated roots modulo $q$, so that
$$
\{(x,y)\in(\Z/q\Z)^2:y^2=x(x^2+a)\}
$$
are the affine points of an elliptic curve. The curve has one point
at infinity, so by the Hasse bound, we have
$$
1+\sum_{x\in\Z/q\Z}\left[1+\left(\frac{x(x^2+a)}{q}\right)\right]
\le q+1+2\sqrt{q},
$$
whence
$$
\sum_{x\in(\Z/q\Z)^\times}\left(\frac{x(x^2+a)}{q}\right)\le 2\sqrt{q}.
$$
This last estimate is less than $q-1$ provided that $q\ge 7$, and we check
the claim for $q\in\{3,5\}$ directly.

Similarly, for $q\ge5$, $x^6+a$ has no repeated roots modulo $q$, so that
$$
\{(x,y)\in(\Z/q\Z)^2:y^2=x^6+a\}
$$
are the affine points of a genus $2$ curve. The curve has two points
at infinity, so by the Weil bound, we have
$$
2+\sum_{x\in\Z/q\Z}\left[1+\left(\frac{x^6+a}{q}\right)\right]
\le q+1+4\sqrt{q},
$$
whence
$$
\sum_{x\in(\Z/q\Z)^\times}\left(\frac{x^6+a}{q}\right)\le
4\sqrt{q}-1-\left(\frac{a}{q}\right)\le 4\sqrt{q}.
$$
This last estimate is less than $q-1$ provided that $q\ge 19$, and we check
the claim for $q\in\{3,5,11,17\}$ directly.
\end{proof}

Theorem~\ref{t:main} follows by induction from the following proposition.
\begin{proposition}
Let $P$ be a finite set of prime numbers and $q$ the smallest prime
not contained in $P$. Then there is a generalized Euclid sequence with
seed $P$ that contains $q$.
\end{proposition}
\begin{proof}
Suppose that $P=\{p_1,\ldots,p_k\}$, and put $n=p_1\cdots p_k$.
If $q=2$ then $n+1$ is even, so we may choose $2$ as the next term,
$p_{k+1}$. Hence we may assume that $q$ is odd.

Put
$$
S=\{d+q\Z:d\in\Z_{>0},\;d\mid n\}\subseteq(\Z/q\Z)^\times,
$$
and note that $S\supseteq S_q$. Suppose first that
$S=(\Z/q\Z)^\times$. If $\left(\frac{-n}{q}\right)=1$ then it
follows that there is a $d\mid n$ such that $d+n/d\equiv0\pmod*{q}$,
so we can choose $q$ as the next term. On the other hand, if
$\left(\frac{-n}{q}\right)=-1$ then by Lemma~\ref{l:hyp}(i) we may choose
$d\mid n$ such that $\left(\frac{d+n/d}{q}\right)=-1$, provided that
$q\ne5$ or $n\not\equiv3\pmod*{5}$. For this choice of $d$ there must be a
prime $p\mid(d+n/d)$ such that $\left(\frac{p}{q}\right)=-1$. Choosing
this $p$ as the next term, we replace $n$ by $n'=pn$, so that
$\left(\frac{-n'}{p}\right)=1$, and we may then follow this by
$q$, as above. For $q=5$ and $n\equiv3\pmod*{5}$ we choose $d=1$;
since $n+1\equiv-1\pmod*{5}$ there is a prime $p\mid(n+1)$ with
$p\not\equiv1\pmod*{5}$, and replacing $n$ by $pn$ gives a different
residue with which we can carry out the proof above.

Suppose now that $S\ne(\Z/q\Z)^\times$. We seek to enlarge $S$ by
continuing the sequence, i.e.\ we choose $p=p_{k+1}$ from
$$
T=\{p:p\text{ prime and }p\mid(d+n/d)\text{ for some }d\mid n\},
$$
and replace $P$ by $P\cup\{p\}$, $n$ by $pn$ and $S$ by $S\cup pS$.
We are free to repeat this procedure until either $q\in T$ (in which
case we may choose $q$ as the next term) or $S$ stabilizes, so that
$pS\subseteq S$ for every choice of $p\in T$.  If that is the case then
it is easy to see that for every $s\in S$, $S$ contains the coset $sG$,
where $G\le(\Z/q\Z)^\times$ is the subgroup generated by $\{p+q\Z:p\in
T\}$. Thus, $S=\bigcup_{s\in S}sG$ is a union of cosets; in particular,
$\#G$ divides $\#S$.

Next, let $H$ be a subgroup of $(\Z/q\Z)^\times$ of index at least $4$.
For any $h\in H$, the number of $d\in(\Z/q\Z)^\times$ such that
$d+n/d=h$ is at most $2$. Hence,
$$
\#\{d\in(\Z/q\Z)^\times:d+n/d\in H\}\le 2\#H\le\tfrac12(q-1).
$$
By Lemma~\ref{l:Sq},
it follows that there exists $d\mid n$ such that $(d+n/d)+q\Z\notin H$.
In turn this implies that $p+q\Z\notin H$ for some $p\in T$.

For any $r\mid(q-1)$, consider the subgroup
$$
H_r=\{h\in(\Z/q\Z)^\times:h^{\frac{q-1}{r}}=1\}
=\{x^r:x\in(\Z/q\Z)^\times\}.
$$
Let $q-1=\prod_{i=1}^mr_i^{e_i}$ be the prime factorization of $q-1$.
For each $r_i\ge5$ we apply the above argument with
$$
H=H_{r_i}=\{h\in(\Z/q\Z)^\times:
r_i^{e_i}\text{ does not divide the order of }h\}
$$
to see that $G$ has order divisible by $r_i^{e_i}$.  For $r_i\in\{2,3\}$
the index of $H_{r_i}$ is too small to apply the argument, but we may
still apply it to $H_{r_i^2}$ (when $r_i^2\mid(q-1)$) to see that $G$
has order divisible by $r_i^{e_i-1}$. Thus we find that the index of $G$
in $(\Z/q\Z)^\times$ divides $6$.

If $q\not\equiv1\pmod*{3}$ then $G$ has index at most $2$, so that
$\frac12(q-1)\mid\#G\mid\#S$; by Lemma~\ref{l:Sq} it follows that
$S=(\Z/q\Z)^\times$, as desired.
If $q\equiv1\pmod*{3}$ then we apply the
above argument with $H=H_6$ to see that there exists $p\in T$ such
that $p^{\frac{q-1}{6}}\not\equiv1\pmod*{q}$. Since
$p^{\frac{q-1}{6}}=p^{\frac{q-1}{2}}/p^{\frac{q-1}{3}}$, it follows that
at least one of $H_2$ and $H_3$ does not contain $p+q\Z$. If
$p+q\Z\notin H_3$ then again $G$ has index at most $2$, and we
conclude that $S=(\Z/q\Z)^\times$ as above.

Hence, we may assume that $p+q\Z\notin H_2$,
so that $G$ has index dividing $3$.
If $S=(\Z/q\Z)^\times$ then we are finished, so we may assume
that $G=H_3$ and $\#S<q-1$. By Lemma~\ref{l:Sq}, we must have $\#S>\#G$,
and it follows that $S=G\cup sG$ for some $s\in(\Z/q\Z)^\times\setminus G$.
Going through the argument above with $H=H_3$, to avoid concluding that
there exists $p\in T$ such that $p+q\Z\notin H_3$, the function
$d\mapsto d+n/d$ must map $S$ 2--1 onto $H_3$. By the quadratic formula,
this means in particular that
$\left(\frac{h^2-4n}{q}\right)=1$ for every $h\in H_3$, and thus
$\left(\frac{x^6-4n}{q}\right)=1$ for every $x\in(\Z/q\Z)^\times$.
However, that contradicts Lemma~\ref{l:hyp}(ii) for
$q\notin\{7,13\}$, and for $q\in\{7,13\}$ we verify directly
that $\#S_q>\frac23(q-1)$. This concludes the proof.
\end{proof}
\bibliographystyle{amsplain}
\providecommand{\bysame}{\leavevmode\hbox to3em{\hrulefill}\thinspace}
\providecommand{\MR}{\relax\ifhmode\unskip\space\fi MR }
\providecommand{\MRhref}[2]{%
  \href{http://www.ams.org/mathscinet-getitem?mr=#1}{#2}
}
\providecommand{\href}[2]{#2}

\end{document}